\documentclass{amsart}
\usepackage{amsfonts,amssymb,amscd,amsmath,enumerate,verbatim,calc,xypic}
\usepackage[all]{xy} \SelectTips{eu}{}
\def\ds{\displaystyle}

\newcommand{\bsx}{{\boldsymbol x}}

\newcommand{\BN}{{\mathbb N}}

\newcommand{\BZ}{{\mathbb Z}}

\newcommand{\fm}{{\mathfrak m}}

\newcommand{\loe}{{{ll}}}

\newcommand{\Ap}{\operatorname{Ap}}

\newcommand{\pd}[2]{\operatorname{pd}_{#1}{#2}}

\newcommand{\depth}[2]{\operatorname{depth}_{#1}{#2}}
\newcommand{\edim}{\operatorname{edim}}

\newcommand{\inde}{\operatorname{index}}

\newcommand{\mult}{\operatorname{mult}}

\newcommand{\codim}{\operatorname{codim}}

\newcommand{\ord}{\operatorname{ord}}

\theoremstyle{plain}
\newtheorem{theorem}{Theorem}[section]
\newtheorem{corollary}[theorem]{Corollary}
\newtheorem{lemma}[theorem]{Lemma}

\newtheorem{proposition}[theorem]{Proposition}

\theoremstyle{definition}

\newtheorem{chunk}[theorem]{}

\newtheorem{remark}[theorem]{Remark}

\newtheorem{example}[theorem]{Example}

\numberwithin{equation}{theorem}

\theoremstyle{remark}

\begin{document}

\date{\today}
\title[The index of a numerical semigroup ring]{The index of a numerical semigroup ring}

\author[O.~Veliche]{Oana~Veliche}

\address{Department of Mathematics, Northeastern University, 360 Huntington Avenue, Boston, MA 02115 }

\email{o.veliche@neu.edu}

\urladdr{http://www.math.northeastern.edu/$\sim$veliche}

\address{Saint Herman Christian School, 62 Harvard Avenue, Allston, MA 02134}
\email{oana@sainthermanschool.org}



\begin{abstract} Let $R=k[|t^a,t^b,t^c|]$  be a  complete intersection numerical semigroup ring over an infinite field $k$, where $a,b,c\in\BN$. The generalized Loewy length, which is Auslander's index in this case, is computed in terms of the minimal generators of the semigroup: $a,b$ and $c$.
Examples  provided show that the left hand side of Ding's inequality  $\mult(R)-\inde(R)-\codim(R)+1\geq 0$  can be made arbitrarily large for  rings $R$  with $\edim(R)=3$ . The index of a complete intersection numerical semigroup ring with embedding dimension greater than three is also computed. 

\end{abstract}

\maketitle

\section*{Introduction}

Let  $(R,\fm, k)$ be  a local noetherian Gorenstein ring with maximal ideal $\fm$ and residue field $k$ and let $M$ be a finitely generated $R$-module. 
{\it Auslander's delta invariant} of the module $M$, denoted by $\delta(M)$, is the smallest non-negative integer $\mu$ such that there exists an exact sequence of $R$-modules, called Cohen-Macaulay approximation,  $0\to Y\to X\oplus R^\mu\to M\to 0$
such that $X$ is a maximal Cohen-Macaulay module with no free direct summands and $\pd{R}{Y}<\infty$; see \cite{ab}.
It is clear from the definition that $\delta(R^n)=n$ for every integer $n\geq 1$. If $\pd RM<\infty$, then $\delta(M)$ is the minimal numbers of generators of the module $M$. Moreover, 
 a surjective homomorphism $M\to N\to 0$ induces an inequality between the delta invariants$\colon$ $\delta(N)\leq\delta(M)$. In particular,  
 $0\leq \delta(R/\fm)\leq \delta(R/\fm^2)\leq\dots\leq\delta(R/\fm^i)\leq\delta(R/\fm^{i+1})\leq \cdots\leq1,$ for all  $i\geq 1$.
		
The {\it index} of the ring $R$ introduced by Auslander and studied by Ding in his thesis \cite{d}, denoted by $\inde(R),$ is defined as the minimum $i\geq 1$ such that $\delta(R/\fm^i)=1$; see also \cite[p. 247]{d92}. The index is finite for all Gorenstein rings.
Ding studies further in  \cite{d,d92} the properties of the index over Gorenstein rings with infinite residue field. He proves that a ring $R$ is regular if and only if $\inde(R)$=1. 
Moreover, he shows that the ring $R$ is a hypersurface if and only if $\inde(R)=\mult(R)$, where $\mult(R)$ the multiplicity of the ring $R$. Furthermore, if $R$ is not regular, then 
\begin{equation*} 
\tag{*} \mult(R)-\inde(R)-\codim(R)+1\geq 0.
\end{equation*}
\noindent
In particular, $\mult(R)\geq\inde(R)$. Here $\codim(R)= \edim(R)-\dim(R)$  denotes the codimension of the ring $R$. For a proof see e.g. \cite[Proposition 1.6]{d92}.

 Martsinkovsky \cite{m96} extends the notion of index  to rings which are not necessarily noetherian, local or Gorenstein. He shows that the index is finite if  $R$ is  a noetherian local ring.  This index satisfies all of the properties mentioned above.

By Ding's result \cite[Theorem 2.14]{d} when $R$ is a hypersurface the left hand side of the inequality $(*)$ is zero. Thus, one could ask the following$\colon$

\vspace{0.3cm}

\noindent 
{\bf Question.}  When $R$ is complete intersection ring is the  left hand side of the inequality $(*)$ bounded above by a constant? 

\vspace{0.3cm}
\noindent
We answer negatively this question by providing examples which show that the left hand side can be made arbitrarily large.
Example \ref{ci example 1} shows that  for any integer  $n$ with  $n\geq 2$,  there exist complete intersection rings $R_n$  with $\edim(R_n)=3$ and
$$\mult(R_n)-\inde(R_n)-\codim(R_n)+1=2n-3.$$
Proposition  \ref{index exp 2} shows that  for any positive integer $n$ there exist complete intersection rings $R_n$ with $\edim (R_n)=n$  such that 
$$\mult(R_{n})-\inde(R_{n})-\codim(R_{n})+1=2^n-2n.$$
These examples were found among the numerical semigroup rings.
 
Working with semigroup rings requires methods different from those used in defining the index originally. In Section \ref{Numerical Semigroup Rings} we introduce some  notions that are needed in the paper. 

The main result of the paper is Theorem \ref{main}. It leads to an explicit formula for the index of  complete intersection numerical rings  semigroup rings $R=k[|t^a,t^b,t^c|]$, where $k$ is an infinite field and $a,b,c\in\BN$, with $\gcd(a,b,c)=1$; see Remark \ref{Remark main}.

 Proposition \ref{Shen-Bryant} in Section \ref{ci edim>3} is a result confirmed by Bryant and Shen \cite{bs}; the proof they suggest is included here. Corollary \ref{index Gor} gives a general formula for the index of a Gorenstein numerical semigroup ring in terms of the order of certain elements of the semigroup. This corollary does not give an explicit formula for the index as given by  Theorem \ref{main}, however its strength is illustrated in the proof of Proposition \ref{index exp 2}
where we compute the index for some rank 1 numerical semigroup rings with
embedding dimension greater than three; see also Remarks \ref{Remark} and \ref{Remark examples}.


\section{Numerical Semigroup Rings}
\label{Numerical Semigroup Rings}

In this section we introduce some terminology that will be used in this paper. 

\noindent
Let $H$ be a semigroup minimally generated by natural numbers $a_1<a_2<\cdots<a_e$. 
The set  $\BN\setminus H$ is finite  if and only if  $\gcd(a_1,a_2,\dots, a_e)=1$; in this case $H$ is called a {\it numerical semigroup} and the {\it Frobenius number} of  $H$ is given by
$$f(H)=\max\{h\not\in H\ |\ h+i\in H,\ \text{for every positive integer}\ i\}.$$
Note that $f(\BN)=-1$, and if $H\not=\BN$, then $f(H)=\max\{\BN\setminus H\}.$

\chunk
\label{Properties semigroup ring}
The associated {\it numerical semigroup ring} of the semigroup $H$ is defined as  $$R=k[|H|]=k[|t^s|s\in H|]$$ and has the following properties:
$R$ is a local ring with maximal ideal $\fm=(t^{a_1},\dots, t^{a_e})$, thus $\edim(R)=e$; $R$ is a domain; $\depth{}(R)=1=\dim{}(R)$;   $\mult(R)=a_1$; and  $R$ is Gorenstein if and only if the semigroup $H$ is symmetric (i.e. $s\in H$ if and only if  $f(H)-s\not\in H.$)

\vspace{0.3cm}

 A semigroup $H$ is called {\it complete intersection} or {\it Gorenstein} if the corresponding semigroup ring $R=k[|H|]$ is complete intersection or Gorenstein respectively.

\chunk
\label{index min} Let $R=k[|H|]$ be a  Gorenstein numerical semigroup ring   with $k$ infinite field. The index of $R$ is well defined and Watanabe shows in \cite[Proposition 1.23]{d}  that
$$
\inde(R)=\min\{i| \fm^i\subseteq (t^s), \ \text{for some}\ s\in H\}.
$$
In particular, if  we set $N_{s}=\min\{i|\fm^i\subseteq (t^{s})\}$,  for $s\in H$,  then 
$$\inde(R)=\min_{1\leq j\leq e} \{N_{a_j}\}.$$

\chunk If $(R,\fm)$ is a noetherian local ring, {\it the generalized Loewy length} is defined by
$$\loe(R)=\min\left\{i|\fm^i\subseteq (\bsx),\ \text{for some system of parameters $\bsx$ of $R$}\right\}.$$  
We recall that for any Gorenstein ring $R$ we have $$\inde(R)\leq\loe(R).$$ Indeed, if $n=\loe(R)$, then $\fm^n\subseteq (\bsx)$ for some system of parameters $\bsx$ of $R$.  This inclusion induces a  surjection $R/{\fm^n}\to R/(\bsx)$ which further induces the inequality  $\delta(R/{\fm^n})\geq \delta(R/(\bsx))$. Since $\bsx$ is a system of parameters, we have $\pd{R}{R/(\bsx)}<\infty$, so $\delta(R/(\bsx))=1$. Thus, $\delta(R/\fm^n)\not=0$, and by definition we have $\inde(R)\leq n$.

Watanabe's result from \ref{index min} shows that when $R$ is a Gorenstein numerical semigroup ring we have $\loe(R)\leq\inde(R).$  Therefore, the index of a Gorenstein numerical semigroup ring is given by its generalized Loewy length.


\section{The index of Gorenstein semigroup rings of edim 3}
\label{Sec: ci edim=3}

The normal semigroup rings of the form $R=k[|t^a,t^b|]$ are hypersurfaces and $\inde(R)=\mult(R)=\min\{a,b\}$. Therefore, we turn our attention to the case when the semigroup ring is of embedding dimension three. These rings were studied by Herzog in \cite{h} and by Watanabe in \cite[Proposition 3]{kw}. 

\chunk 
\label{ci edim 3}
Let  $H$ be a numerical semigroup minimally generated by three elements and set $R=k[|H|]$ for an infinite field $k$. The following are equivalent, after  a possible relabeling of the generators of $H$.
\begin{enumerate}[\rm\quad(i)]
\item $R$ is complete intersection;
\item $R$ is Gorenstein;
\item There exist integers $p,x,y\geq 2$ such that $H=\langle a,b,c\rangle$ where $$  a\in\langle x,y\rangle\  \text{with}\  a\not\in\{x,y\}\ \text{and}\ b=px,\ c=py,\ \gcd(x,y)=\gcd(a,p)=1.$$ 
\end{enumerate} 
Moreover, when one (hence all) of these cases holds, $f(H)=pxy+pa-(a+b+c).$

\begin{theorem}
\label{main} 
Let $R$ be a complete intersection semigroup ring of embedding dimension three. With the notation from \ref{ci edim 3}, set  $a=a'x+a''y$ with $a'$ and $a''$ non-negative integers. Then the following equalities  hold.
	
	\begin{enumerate}[\rm\quad(a)]
	
		\vspace{0.5cm}

		\item If $x<y$, then $N_a=
			\begin{cases}
				x+a'+y\ds\frac{a''}{x}-1,&\text{if}\ \ds\frac{a''}{x}\in\BN;\\
				&\\
				y+a'+a''+(y-x)\ds\left\lfloor\frac{a''}{x}\right\rfloor-1,&\text{if}\ 
				\ds\frac{a''}{x}\not\in\BN.
			\end{cases}$
		
		\vspace{0.5cm}
		
		\item $N_b=
			\begin{cases}
				p+x-1,&\text{if}\ a'\not=0\\
				     &\text{or}\ a'=0,\ \text{and}\  p<a''; \\
				     &\\
				a''+p\ds\frac{x}{a''}-1,&\text{if}\ a'=0,\ a''<p,\ \text{and}\ 
				\ds\frac{x}{a''}\in\BN;\\ 
				&\\    
				p+x-1+(p-a'')\left\lfloor\ds\frac{x}{a''}\right\rfloor,&\text{if}\ a'=0,\ 
				a''<p,\ \text{and}\ \ds\frac{x}{a''}\not\in\BN.
			 \end{cases}$
		 
		\vspace{0.5cm} 
		 		 
		\item $N_c=
			\begin{cases}
				p+y-1,&\text{if}\ a''\not=0\\
				     &\text{or}\ a''=0,\ \text{and}\  p<a'; \\
				     &\\
				a'+p\ds\frac{y}{a'}-1,&\text{if}\ a''=0,\ a'<p,\ \text{and}\  	 
				\ds\frac{y}{a'}\in\BN;\\ 
				&\\    
				p+y-1+(p-a')\left\lfloor\ds\frac{y}{a'}\right\rfloor,&\text{if}\ a''=0,\ 
				a'<p,\ \text{and}\ \ds\frac{y}{a'}\not\in\BN.
			 \end{cases}$
		 	\end{enumerate}
\end{theorem}

\begin{proof} 

	(a). Set $N=N_a$. We may reduce to the case $a''<x$ and 
	prove that
	$$N=
	\begin{cases}
		x+a'-1,&\text{if}\ a''=0;\\
		y+a'+a''-1,&\text{if}\ a''\not=0.           
	\end{cases}  
	$$
	
	Indeed, assume that $a=a'x+a''y$ for some non-negative integers $a'$ and $a''$. There exist unique non-negative integers $q$ and $d''$ such that
	$a''=qx+d''$ with $0\leq d''<x$. If we set $d'=a'+qy$, then  $a=d'x+d''y$.
We have $d''=0$ if and only if $\ds\frac{a''}{x}\in\BN$, and then 
	$$x+d'-1=x+a'+y\ds\frac{a''}{x}-1.$$
We have $d''\not=0$ if and only if $\ds\frac{a''}{x}\not\in\BN$, and then
	$$y+d'+d''-1=y+a'+a''+(y-x)\ds\left\lfloor\frac{a''}{x}\right\rfloor-1.$$
For the rest of the proof we thus assume that $0\leq a''<x<y$.
By definition, $N$ is the minimum natural number with the property that for any non-negative integers $u,v$ and $w$ such that $u+v+w=N$ there exist non-negative integers $\alpha,\beta$ and $\gamma$ with $\alpha\geq 1$ such that 
\begin{equation*}
		\label{eq}
		ua+vb+wc=\alpha a+\beta b +\gamma c.
	\end{equation*}
We may assume that $u=0$, thus we get 
\begin{equation*}
		vpx+wpy=\alpha a+\beta px +\gamma py.
	\end{equation*}
Since $\gcd(a,p)=1$, there exists an integer $\alpha'\geq 1$ such that $\alpha=p\alpha'$, so after dividing by $p$ the equality above becomes 
	\begin{align*}
		vx+wy&=\alpha'(a'x+a''y)+\beta x+\gamma y\ \iff\\
		x(v-\alpha'a'-\beta)&=y(\gamma +\alpha'a''-w).
	\end{align*}
Since $\gcd(x,y)=1$, there exists an integer $z$ such that
\begin{align*}
		v-\alpha'a'-\beta&=yz\\
		\gamma+\alpha'a''-w&=xz.
	\end{align*}
Setting $v=N-\delta$ and $w=\delta$ for some $\delta\in\{0,1,\dots,N\}$, and using that $\alpha'\geq 1,\ \beta\geq 0$ and $\gamma\geq 0$, we get that $N$ is the  minimum positive integer such that for each $0\leq \delta\leq N$, there exists $z_\delta\in\BZ$ such that
\begin{equation}
		\label{z inequality}
		\frac{N-\delta-a'}{y}\geq z_\delta\geq \frac{-\delta+a''}{x}.
	\end{equation}  
In particular, for $\delta=0$ we have 
\begin{equation}
		\frac{N-a'}{y}\geq z_0\geq \frac{a''}{x}.
	\end{equation}

\noindent	
Case $a''=0$. 
	
	The inequality (\ref{z inequality}) becomes
	
	\begin{equation}
		\label{z inequality a''=0}
		\frac{N-\delta-a'}{y}\geq z_\delta\geq -\frac{\delta}{x}.
		\end{equation} 
If we assume that $N<x$, then by choosing $\delta=N$ the inequality above implies
		\begin{equation*}
		0>-\frac{a'}{y}\geq z_N\geq-\frac{N}{x}>-1,
	\end{equation*}
which cannot happen for any $z_N\in\BZ$. Thus, $N\geq x$.
	
\noindent	
	If $0\leq\delta< x$, then $z_\delta\ge0$, thus $N\geq\delta +a'$. In particular, when $\delta=x-1$ we get
	\begin{equation}
	\label{N ineq}
	N\geq x+a'-1.
	\end{equation}
	\noindent
	If $\delta \geq x$, write $\delta=qx+r$ with $q,r$ non-negative integers with $0\leq r<x$. Using (\ref{N ineq}) we obtain
\begin{align*}
	                                         \frac{N-\delta-a'}{y}+\frac{\delta}{x}                   &=\frac{N-a'}{y}+\delta\frac{y-x}{xy}\\
	                                   &\geq\frac{x-1}{y}+\delta\frac{y-x}{xy}\\
	                                   &\geq\frac{x-1}{y}+\frac{y-x}{y}\\
	                                   &=1-\frac{1}{y}\\
	                                   &>1-\frac{1}{x}\\
	                                   &=\frac{x-1}{x}\\
	                                   &\geq\frac{r}{x}. 
	\end{align*}
So, $\ds\frac{N-\delta-a'}{y}\ge\frac{r-\delta}{x}=-q$.	In particular, the inequality (\ref{z inequality a''=0}) holds if  $z_\delta=-q$.
	Thus, we have $N=x+a'-1.$

\noindent
Case $a''\not=0$.
	
	If we assume that $N<a''$, then by choosing $\delta=N$, the inequality (\ref{z inequality}) implies 
\begin{equation*}
		0\geq-\frac{a'}{y}\geq z_N\geq\frac{-N+a''}{x}>0,
	\end{equation*}
which is a contradiction. Thus,  we must have $N\geq a''$. We apply inequality (\ref{z inequality}) to several cases of $\delta$ in order to show that we should have $N=a'+a''+y-1$.
	
	\noindent
	If $\delta=a''$, then $z_\delta\geq 0$ and then $N\geq a'+a''$.
	
	\noindent
	If $0\leq \delta<a''$, then $z_\delta\geq 1$ and  then $N\geq a'+y+\delta$. In particular, for $\delta=a''-1$, $$N\geq a'+a''+y-1.$$ 
	
	\noindent
	If $\delta>a''$, then write $\delta-a''=qx+r$ for non-negative integers $q,\ r$ and $0\leq r<x$. Using that $N\geq a'+a''+y-1$ and $\delta\geq a''+1$, we obtain
	
	\begin{align*}
	\frac{N-\delta-a'}{y}+\frac{\delta}{x}
		    &=\frac{N-a'}{y}+\delta\frac{y-x}{xy}\\
				&\geq\frac{a''+y-1}{y}+(a''+1)\left(\frac{1}{x}-\frac{1}{y}\right)\\
		    &=1+\frac{1}{x}-\frac{2}{y}+\frac{a''}{x}\\
				&=\frac{x-1}{x}+2\left(\frac{1}{x}-\frac{1}{y}\right)+\frac{a''}{x}\\
		    &>\frac{r}{x}+\frac{a''}{x}.
	\end{align*}
So, $\ds\frac{N-\delta-a'}{y}\geq\frac{r-\delta}{x}+\frac{a''}{x}=-q+\frac{a''}{x}$.	Then  the inequality (\ref{z inequality}) holds by taking $z_\delta=-q$.
	Therefore, $N=a'+a''+y-1$.
	
	(b). We consider two cases.
	
\noindent
	 Case $a'\not=0$. 
	
	First, we show that $\fm^{p+x-1}\subseteq (t^b)$. This is equivalent with showing that 
	for any $u,v$ and $w$ non-negative integers such that $u+v+w=p+x-1$ there exist non-negative integers $\alpha,\beta$ and $\gamma$ with  $\beta\not=0$ such that
	\begin{equation*}
		ua+vb+wc=\alpha a+\beta b +\gamma c.
	\end{equation*} 
If $v\not=0$, then this is clear. Assume that $v=0$, $u=p+x-1-\delta$ and $w=\delta$ where $\delta\in\{0,\dots ,p+x-1\}$.
	We consider the two cases: $\delta\leq x-1$ and $\delta>x-1$.

\noindent	
If $\delta\leq x-1$, we have: 
\begin{align*}
		ua+wc&=(p+x-1-\delta)a+\delta py\\
		&=(x-1-\delta)a+ a'px+(\delta+a'')py\\
		&=(x-1-\delta)a+a'b+(\delta+a'')c.
	\end{align*} 
Since $a'\not=0$, we have written $ua+wc$ in the desired format.

\noindent	
If $\delta>x-1$, we set $\theta=\delta-x+1>0$. Remark that $p-\theta=u\geq 0$. We have:
\begin{align*}
		ua+wc&=(p+x-1-\delta)a+\delta py\\
		&=(p-\theta)a+ pxy+(\theta-1)py\\
		&=(p-\theta)a+yb+(\theta-1)c.
	\end{align*} 
Next, we show that $\fm^{p+x-2}\not\subseteq (t^b)$. Assume that there exist non-negative integers  $\alpha, \beta$ and $\gamma$ with $\beta\not=0$  and  $(p-1)a+(x-1)c=\alpha a+\beta b+\gamma c.$ This is  equivalent to
\begin{equation}
\label{equation pa+}		
		     pa+(x-1)py=(\alpha+1)a+\beta px+\gamma py.  
	\end{equation}                                                               
Since $\gcd(a,p)=1$, there exists a positive integer $\alpha'$ such that $\alpha+1=p\alpha'$. Thus, the last equality above is equivalent to
\begin{align*}
		a'x+a''y+(x-1)y&=\alpha'(a'x+a''y)+\beta x+\gamma y\ \iff \\
		y(a''-1+x-\alpha'a''-\gamma)&=x(\alpha'a'+\beta-a')\ \iff \\
		y[x-a''(\alpha'-1)-\gamma]&=x[a'(\alpha'-1)+\beta].
	\end{align*}                                     
Since $\gcd(x,y)=1$ and $a'(\alpha'-1)+\beta>0$, there exists a  positive integer $z$ such that 
\begin{align*}
		x-a''(\alpha'-1)-\gamma&=xz.
	\end{align*}
This together with the fact that $a''(\alpha'-1)+\gamma\geq 0$, implies that $z=1$ thus
\begin{align*}
		a''(\alpha'-1)=0\ \text{and}\  \gamma=0.
	\end{align*} 
If $a''\not=0$, then $\alpha'=1$ and, after dividing by $p$, the equality (\ref{equation pa+}) becomes $(x-1)y=\beta x$, which is a contradiction since $\gcd(x,y)=1$.
	If $a''=0$, then $a=a'x$ and, after dividing by $p$, the equality (\ref{equation pa+}) becomes $(x-1)y=(\alpha'-1)a'x+\beta x$, which is again a contradiction.

\noindent	
Case $a'=0$.  In particular, $a''\not=0$. 
	
	If $p<a''$, then apply part (a)  with 
	$a,\ b,\ c,\ p,\ x,\ y,\ a',\ a''$ taken to be  $px,\ py,\ a''y,\ y,\ p,\ a'',\ x,\ 0$ respectively.
	
	\noindent
	If $p>a''$, then apply part (a)  with 
	$a,\ b,\ c,\ p,\ x,\ y,\ a',\ a''$ taken to be $px,\ a''y,\ py,\ y,\ a'',\ p,\ 0,\ x$ respectively.

	(c) follows from (b) due to the symmetry of the statement in $x$ and $y$.

\end{proof}

\remark 
\label{Remark main}
Theorem \ref{main} and the results from \ref{index min} allow us now to give a ``formula'' for the index of a numerical semigroup ring of embedding dimension three  in terms of the generators of the semigroup. Indeed,
$\inde(R)=\min\{N_a,N_b,N_c\},$ where  $N_a, N_b$, $N_c$ were computed in  Theorem \ref{main}. Below are some special cases.

\begin{corollary}
\label{corollary main}
Let $R$ be a complete intersection semigroup ring of embedding dimension three. With the notation from \ref{ci edim 3}, set  $a=a'x+a''y$ with $a'$ and $a''$ non-negative integers and assume that $a''<x<y$.
\begin{enumerate}[\rm\quad(a)]
\item If $a'\not=0$ and $a''\not=0$, then $$\inde(R)=\min\{y+a'+a''-1,\  p+x-1\}.$$

\item If $a'=0$ and $p<a''$, then $$\inde(R)=\min\{y+a''-1,\  p+x-1\}.$$

\item If $a''=0$, then $$\inde(R)=\min\{ p+x-1,\  x+a'-1\}.$$
\end{enumerate}
\end{corollary}

\begin{example} 
\label{ci example 1}
Let $R=k[|t^{4n}, t^{(4n+1)(2n-1)},t^{(4n+1)(2n+1)}|]$ where $k$ is an infinite  field and $n\geq 2$. Then $R$ is a complete intersection ring with 
$$\mult(R)-\inde(R)-\codim(R)+1=2n-3.$$
\end{example}

\begin{proof}
If we let $a=4n$, $x=2n-1$, $y=2n+1$ and $p=4n+1$, we obtain by Corollary \ref{corollary main}(a) that $\inde(R)=\min\{2n+2, 6n-1\}=2n+2$.
So, using now \ref{Properties semigroup ring} we get $\mult(R)-\inde(R)-\codim(R)+1=4n-(2n+2)-2+1=2n-3.$\end{proof}

\section{The index of a complete intersection  semigroup ring of edim >3}
\label{ci edim>3}

	In this section let  $H=\langle a_1,a_2,\dots,a_e\rangle$ be a semigroup  with 
	$\gcd(a_1,\dots,a_e)=1$ and $a_1<a_2<\dots<a_e$ and set $R=k[|H|]$ for an infinite field $k$.  The {\it order of an element} $s$ of a semigroup $H$, see \cite{b}, is defined by 
$$\ord(s)=\max\left\{\displaystyle\sum_{i=1}^{e}\alpha_i \Bigg| s=\sum_{i=1}^e\alpha_ia_i \right\}.$$

\remark 
\label{Remark} If $R$ is a Gorenstein numerical semigroup ring, then \cite[Lemma 2.5]{s} shows that $N_{a_1}=\ord(f(H)+a_1)+1.$ A direct computation, which we omit here, similar in difficulty and length to the proof of Theorem \ref{main}, shows that when $\edim(R)=3$ we have  $N_{a_i}=\ord(f(H)+a_j)+1$ for all $1\leq j\leq 3$. Prompted by this discovery, the author consulted Bryant and Shen \cite{bs} who confirmed that a more general result holds.   Although the  proof they suggest uses techniques not used in this paper, we include it here as it does not appear in literature. For details  on terminology and background results see \cite{b}.

\begin{proposition} 
\label{Shen-Bryant}
If  $R=k[|H|]$ is a Gorenstein numerical semigroup ring, then $$N_s=\ord(f(H)+s)+1,\quad\text{for all}\quad s\in H\setminus \{0\}.$$
\end{proposition}

\begin{proof}
We recall, see \cite{b}, that the Ap\'ery set of $n\in H\setminus\{0\}$ is defined by $\Ap(H;n)=\{w\in H\mid w-n\not\in H\}$. When $s\in H$, then $$\Ap(H;s)=\{w_0,w_1,\dots,w_{s-1}\},$$
where $0=w_0<w_1<\dots<w_{s-1}=f(H)+s$.

 A homogeneous element $t^w\in\fm^i$ belongs to $(t^s)$ if and only if  $t^w=t^s\cdot t^u$ for some $u\in H$ if and only if $w-s\in H$.  On the other hand, $t^w\in\fm^i$ if and only if $\ord(w)\geq i$. 
Therefore, we get the first equality of: 
\begin{align*}
 N_s&=\min\{i\mid \text{for all}\ w\in H\  \text{such that}\ \ord(w)\geq i,\  \text{we have}\ w\not\in\Ap(H;s)\}\\
 &=\max\{\ord(w)\mid w\in\Ap(H;s)\setminus\{0\}\}+1.
 \end{align*}
 The second equality follows from the definition of the Ap\'ery set.
 It is easy to check that the proof of \cite[Proposition 3.6]{b} holds also  when $a_1$ is replaced by any $s\in H$. Thus,  for  a symmetric semigroup $H$ we obtain
 $$w_i+w_j=w_{s-1}=f(H)+s\ \text{for all}\ i+j=s-1.$$
 In particular, $$\ord(w)\leq\ord(f(H)+s),\ \text{for all}\ w\in\Ap(H;s)\setminus\{0\}.$$
 The desired conclusion now follows.\end{proof}
 
\begin{corollary}
\label{index Gor}
If $R=k[|H|]$ is a Gorenstein numerical semigroup ring, then 
$$\inde(R)=\min_{1\leq j\leq e}\{\ord(f(H)+a_j)\}+1.$$

\end{corollary}

The next example shows that Proposition \ref{Shen-Bryant} does not hold in the case when $R$ is not a Gorenstein ring.

\begin{example} Let $k$ be an infinite  field and  $R=k[|t^4,t^5,t^{11}|]$. If $f$ is the  Frobenius number of the semigroup $\langle 4,5,11\rangle$, then  
$$N_4\not=\ord(f+4)+1,$$
\end{example}
\noindent
Indeed, $f=7$ and $\ord(f+4)+1=\ord(11)+1=2$ and $N_4\not=2$ as $t^{10}\in\fm^2\setminus(t^4)$.

\remark 
\label{Remark examples}
By contrast to Theorem \ref{main}, Corollary \ref{index Gor} does not give a  precise formula for the index in terms of the generators of the semigroup. 
However,  this corollary is very useful in computing the index of inductively defined Gorenstein numerical semigroup rings, at least when their Frobenius number can be explicitly computed.

\vspace{0.5cm}

Watanabe shows how one can construct complete intersection numerical semigroup rings. His result is generalized by Delorme \cite[Proposition 10]{dc}. 

\chunk
\cite[Lemma 1]{kw}
\label{w-exp 2}
If $H$ is a complete intersection numerical semigroup minimally generated by natural numbers $a_1<a_2<\cdots<a_e$, then $H'=\langle a, pH\rangle$ is a complete intersection for all $a$ and $p$ such that $a=\sum_{i=1}^{e}\alpha_ia_i$ with $\sum_{i=1}^{e}\alpha_i>1$ and $\gcd(a,p)=1$.  Moreover, in this case we have
$
f(H')=p\cdot f(H)+(p-1)a.
$

 \chunk
 \cite[Lemma 3]{kw}
\label{exp 2}
Set  $H_{n,a}=\langle 2^n,2^n+a,\dots,2^n+2^ia,\dots, 2^n+2^{n-1}a\rangle$ with $n\geq 1$ and $a$ is a positive odd integer.
The semigroup  $H_{n,a}$ is a complete intersection as it is obtained inductively: $H_{n,a}=\langle 2^n+a, 2H_{n-1,a}\rangle$, for all $n\geq 2$ and $H_{1,a}=\langle 2,2+a\rangle$.

\begin{lemma}
\label{Frobenius exp 2} For $n\geq 1$, the semigroup  $H_{n,a}$ defined in \ref{exp 2} has Frobenius number
$$f(H_{n,a})=(n-1)2^n+(2^n-1)a.$$ 
\end{lemma}
\begin{proof}
It follows by induction on $n$ using \ref{w-exp 2}.
\end{proof}

\begin{proposition} 
\label{index exp 2}
For $n\geq 1$, and $k$ an infinite field, let  $R_{n,a}=k[|H_{n,a}|]$ be a semigroup ring with $H_{n,a}$ given in \ref{exp 2}. Then $$\inde(R_{n,a})=n+1.$$
In particular, $$\mult(R_{n,a})-\inde(R_{n,a})-\codim(R_{n,a})+1=2^n-2n.$$

\end{proposition}

\begin{proof} Using Corollary \ref{index Gor} it is enough to compute the minimum of  $\ord(f(H_{n,a})+2^n)$ and $\ord(f(H_{n,a})+2^n+2^{k}a)$ for all $0\leq k\leq n-1$.

We claim that  $\ord(f(H_{n,a})+2^n)=n.$
Indeed, by Lemma \ref{Frobenius exp 2} we have 
\begin{align*}
f(H_{n,a})+2^n&=n2^n+(2^n-1)a\\
              &=(2^n+a)+(2^n+2a)+\cdots+(2^n+2^{n-1}a).
\end{align*}
Thus, $\ord(f(H_{n,a})+2^n)\geq n.$ In general, we assume that  
$$
f(H_{n,a})+2^n=\alpha_02^n+\alpha_1(2^n+a)+\alpha_2(2^n+2a)+\cdots+\alpha_n(2^n+2^{n-1}a).$$
This is equivalent to 
$$n2^n+(2^n-1)a=\alpha_02^n+\alpha_1(2^n+a)+\alpha_2(2^n+2a)+\cdots+\alpha_n(2^n+2^{n-1}a).$$ Thus, we have
$$
2^n\cdot[n-(\alpha_0+\alpha_1+\cdots+\alpha_n)]=a\cdot(\alpha_1+2\alpha_2+\cdots+2^{n-1}\alpha_n-2^n+1),$$
which implies, since $a$ is odd, that there exists an integer  $z$ such that 
\begin{align*}
n-(\alpha_0+\alpha_1+\cdots+\alpha_n)&=az,\ \text{and}\\
\alpha_1+2\alpha_2+\cdots+2^{n-1}\alpha_n-2^n+1&=2^nz.
\end{align*}
In particular, the second equality gives that $z\geq 0$ and the first one gives that 
$$ \alpha_0+\alpha_1+\cdots+\alpha_n=n-az\leq n.$$
Since the order is given by the maximum of such sums,  the claim holds.

Next, we show that $\ord(f(H_{n,a})+2^n+2^ka)\geq n$ for all $0\leq k\leq n-1$, thus concluding our proof. 

\noindent If $k=0$, then $f(H_{n,a})+2^n+a=(n+a)2^n.$ Thus, $\ord(f(H_{n,a})+2^n+a)\geq n+a$. 

\noindent If $1\leq k\leq n-2$, then 
\begin{align*}
f(H_{n,a})+2^n+2^ka&=n2^n+(2^n-1)a+2^ka\\
                   &=(2^n+a)+\dots+(2^n+2^{k-1}a)\\
                   &+2(2^n+2^{k+1}a)+\dots+(2^n+2^{n-1}a).
\end{align*}
Thus, $\ord(f(H_{n,a})+2^n+2^ka)\geq n$.

\noindent If  $k=n-1$, then 
\begin{align*}
f(H_{n,a})+2^n+2^{n-1}a&=n2^n+(2^n-1)a+2^{n-1}a\\
                   &=(2^n+a)+\dots+(2^n+2^{n-2}a)+ (1+a)2^n.
\end{align*}    
Thus,  $\ord(f(H_{n,a})+2^n+2^{n-1}a)\geq n+a$. Therefore, the minimum order is $n$, so $\inde(R_{n,a})=n+1$. The last equality in the statement follows from \ref{Properties semigroup ring}. \end{proof}


\section*{Acknowledgments}
I would like to thank Alex Martsinkovsky for several discussions on this subject, Lance Bryant  and Yi-Huang Shen for providing the proof of Proposition \ref{Shen-Bryant} and the referee for valuable comments that improved the presentation.


\end{document}